\documentclass[final,1p]{elsarticle}

\usepackage[utf8]{inputenc}
\usepackage[dvipsnames]{xcolor}
\usepackage{amsmath,amssymb,amsthm}
\usepackage{physics}
\usepackage{halloweenmath}
\usepackage{lettrine}
\usepackage{hyperref}
\usepackage[ruled,vlined]{algorithm2e}
\usepackage{subcaption}
\usepackage{bbm}

\usepackage[T1]{fontenc}
\usepackage[english]{babel}
\usepackage{tikz}
\usepackage{pgfplots}

\usepackage{xcolor}

\hypersetup{colorlinks = true, allcolors = blue}
\usepackage[nameinlink,noabbrev,capitalize]{cleveref}

\newcommand{\Real}{\mathbb{R}}

\newcommand{\Prb}{\mathbb{P}}
\newcommand{\Esp}{\mathbb{E}}

\newtheorem{theorem}{Theorem}[section]

\newtheorem{assumption}{Assumption}[section]
\theoremstyle{remark}

\input Typocaps.fd

\definecolor{customOrange}{HTML}{EB8934}
\definecolor{customTeal}{HTML}{538A96}
\definecolor{customBlue}{HTML}{34C9EB}
\definecolor{customBrown}{HTML}{6B5B4D}

\begin{document}

\begin{frontmatter}
\title{Conservative Surrogate Models for Optimization with the Active Subspace Method}

\author[GIREF]{Philippe-André Luneau}
\ead{philippe-andre.luneau.1@ulaval.ca}
\address[GIREF]{Groupe Interdisciplinaire de Recherche en \'El\'ements Finis de l'Universit\'e Laval, D\'epartement de Math\'ematiques et Statistique, Universit\'e Laval, Qu\'ebec, Canada}
\cortext[cor1]{Corresponding author}

\begin{abstract}
    We are interested in building low-dimensional surrogate models to reduce optimization costs, while having theoretical guarantees that the optimum will satisfy the constraints of the full-size model, by making conservative approximations. The surrogate model is constructed using a Gaussian process regression (GPR). To ensure conservativeness, two new approaches are proposed: the first one using bootstrapping, and the second one using concentration inequalities. Those two techniques are based on a stochastic argument and thus will only enforce conservativeness up to a user-defined probability threshold. The method has applications in the context of optimization using the active subspace method for dimensionality reduction of the objective function and the constraints, addressing recorded issues about constraint violations. The resulting algorithms are tested on a toy optimization problem in thermal design.
\end{abstract}


\end{frontmatter}

\section{Introduction}

Optimization of high-fidelity models in engineering and industrial applications is often a complicated task due to the curse of dimensionality. To be sufficiently accurate, the models considered are required to be large-scale and can contain even more than a million variables. Thus, it is highly relevant to study the structure of those optimization problems to see if any variable or region of the search space can be ignored. For instance,~\cite{Cartis_Otemissov_2022,Wang_Hutter_Zoghi_Matheson_De_Freitas_2016} use random embeddings to search for optima in lower-dimensional hyperplanes, and~\cite{Tezzele_Fabris_Sidari_Sicchiero_Rozza_2023} does proper orthogonal decomposition (POD) to reduce the size of the parameter space for a PDE-constrained optimization problem. Another such technique is the active subspace method (ASM)~\cite{Russi_2010}, which restrains the optimization space to a linear manifold encapsulating most of the variation in the quantity of interest. Afterwards, the minimizer in the low-dimensional subspace is pulled back to the full space using regularization techniques~\cite{Lukaczyk_Constantine_Palacios_Alonso_2014}. This method not only is useful for dimensionality reduction in optimization, but also to quantify uncertainty in complex models~\cite{Khatamsaz_Molkeri_Couperthwaite_James_Arróyave_Srivastava_Allaire_2021}. The ASM has often been used in concert with \textit{kriging} or Gaussian process regression (GPR)~\cite{Rasmussen_Williams_2005} to represent the objective on the active subspace to limit computational costs~\cite{Constantine_Dow_Wang_2014}. This entanglement between the two methods has become stronger over the years, as active subspaces have now been used in the context of bayesian optimization using Gaussian processes~\cite{Khatamsaz_Molkeri_Couperthwaite_James_Arróyave_Srivastava_Allaire_2021,Khatamsaz_Allaire_2023}.

In the context of an optimization problem with nonlinear constraints, one could want to apply a dimensionality reduction technique to the constraints as well as to the objective. In~\cite{Lukaczyk_Constantine_Palacios_Alonso_2014} the authors test out such an approach on a geometric optimization problem, which yields interesting results. 
There are two recorded disadvantages to this approach. 
First, the base evaluation cost is very high when compared to applying a regular gradient-based optimizer to the full-size problem, in the case where the reduced model is used only once. The number of evaluations of the objective would be amortized if multiple resolutions were done.
Second, using a GPR approximation for the nonlinear constraint reduces the cost of evaluating it, but no mechanism is put in place so that the exact constraint is actually enforced when the solution to the optimization problem is pulled back to the high-dimensional space (the regularization takes care of bound constraints, but not general nonlinear constraints). They present in \cite{Lukaczyk_Alonso_Constantine_Kochenderfer_Kroo_2015} an updated optimization strategy, driven by an acquisition function taking into account the probability of feasibility of a point, but violation of the constraint could still occur in some applications.

Of those two fallbacks, this work aims to address the latter. Given the optimization problem
\begin{align}
    \min_{x\in X} F(x) \qq{s.t.} G(x) \leq 0,
    \label{eq:nlproblem}
\end{align}
we want to approximate the general nonlinear constraint by a GPR over an active subspace, $\Tilde{G}$, called a \textit{surrogate model} for $G$. A surrogate model is said to be \textit{conservative} if it always overestimates the function it approximates, \textit{i.e.} $G \leq \Tilde{G}$. By ensuring $\Tilde{G}$ is conservative, a solution to the optimization problem with the constraint $\Tilde{G}\leq 0 $ will 
be an admissible solution according to the exact constraint: 
$$\left\{x\in X\,|\, \tilde{G}(x)\le 0\right\} \subset \left\{x\in X\,|\, G(x)\le 0\right\}.$$
Conservative approximations have been studied in the past in many contexts, such as structural mechanics~\cite{Picheny_Kim_Haftka_Peters_2006,Picheny_Kim_Haftka_Queipo_2008, Zhao_Choi_Lee_Gorsich_2013}, aerodynamics \cite{Lukaczyk_Alonso_Constantine_Kochenderfer_Kroo_2015}, and more recently, in hydrogeology~\cite{Zhang_Qiang_Liu_Zhu_Lv_2022} and general risk management~\cite{Jakeman_Kouri_Huerta_2021}. They are useful in many contexts where it is critical not to underestimate quantities of interest. Different techniques have been tried to ensure conservativeness (\textit{e.g.} bootstrapping~\cite{Picheny_Kim_Haftka_Peters_2006}, Akaike information criterion~\cite{Zhao_Choi_Lee_Gorsich_2013}). Most of these approaches rely on adding an artificial bias to the training data of the surrogate model to shift it upwards, making it conservative on a larger subset of the optimization space (or the whole space). This work will follow this trend, and propose two ways to choose a bias ensuring that the surrogate model is conservative \textit{with high probability}. Those new techniques are very simple to implement, as they are based on elementary iterative methods to select the bias needed to obtain a given probability of conservativeness before the optimization process starts. A short presentation of the ASM and GPR will be done, followed by the analysis of the proposed biased ASM, and an application of the algorithm to a toy problem using the library ATHENA~\cite{Romor_Tezzele_Rozza_2021}.

\section{Theoretical setting}

 The approach to ensure $\tilde{G}$ is conservative involves imposing a bias $\beta$ on the training data of the model, as in~\cite{Picheny_Kim_Haftka_Queipo_2008,Jakeman_Kouri_Huerta_2021,Zhang_Qiang_Liu_Zhu_Lv_2022}. Our proposed method consists of two steps:
\begin{enumerate}
    \item The bias will be selected high enough such that the expected value of the difference between the model and the real constraint is at least $\varepsilon>0$, making it conservative \textit{on average};
    \item The bias can be ajusted iteratively until the probability of this difference being $\varepsilon$ away from its expected value is arbitrarily high, using concentration inequalities as a criteria.
\end{enumerate}
Another method using bootstrap \cite{Givens_Hoeting_2012}, inspired by~\cite{Picheny_Kim_Haftka_Peters_2006}, is also presented, and is shown to be much more precise and efficient in a simple case.

When sampling the GPR of the constraint, the value will then be above the exact constraint with a probability as high as desirable. Of course, increasing the bias too much could also make the model far off from the real constraint. This has to be taken into account in any analysis using this method.

Since the numerical models used to evaluate the constraint functions are often very expensive (numerical simulations), a special attention is given to limit as much as possible additional sampling for the training and bootstrap.

\subsection{Active Subspace Method (ASM) with kriging}
The active subspace method relies on four layers of approximations to approximate a function $f:X\subset \Real^D \to \Real$, namely
\begin{align*}
    f(x) \approx f_\mathrm{AS} \approx f_\mathrm{MC} \approx f_{GPR} \approx f_{NGPR}.
\end{align*}
$f_\mathrm{AS}$ will represent $f$ averaged over the inactive subspace. $f_\mathrm{MC}$ represents a Monte Carlo (MC) estimation of $f_\mathrm{AS}$, necessary because of the complexity of the underlying probability measure. Then, $f_\mathrm{GPR}$ will be a GPR interpolation of $f_\mathrm{MC}$, to reduce the computational cost of repeated MC simulations when evaluating the function. Finally, $f_\mathrm{NGPR}$ will induce variance in the surrogate to prevent interpolation of inexact data caused by the successive approximations.

The most fundamental step in ASM is the first approximation, which effectively reduces the dimensionality of the problem. First, $M$ points are sampled in $X$ according to a probability density $\rho$ over $X$. An approximate covariance matrix is approximated using a Monte Carlo (MC) integral estimator
\begin{align*}
    \Sigma_f &= \int_X \nabla f(x) \otimes \nabla f(x) \dd x \\
    &\approx \frac{1}{M} \sum_{i=1}^M \nabla f(x_i) \otimes \nabla f(x_i) = C,
\end{align*}
as an average of rank-one tensors representing local variations in $f$. The eigenvalues of $C$ converge to the eigenvalues of the real covariance matrix $\Sigma_f$ as $M$ grows larger~\citep{Constantine_2015}. Using numerical singular value decomposition (SVD), a principal component analysis of $C$ can be made by diagonalization
\begin{align*}
    C\approx W^\top\Lambda W.
\end{align*}
By choosing the $d<D$ first columns of $W$, one can capture $100\times(\sum_{i=1}^d \lambda_i / \sum_{i=1}^D\lambda_i)\%$ of the variance in $f$. Let this projection be noted $W_1$, and $W = \mqty[W_1, W_2]$.

To approximate $f$ on the low-dimensional subspace $\mathcal{Y} = \qty{y\in\Real^d : y = W_1^\top x, x \in X}$, $f_\mathrm{AS}(y)$ will be given by the average value of $f$ at a given $y$ value, \textit{i.e.} $f_\mathrm{AS}(y) = \Esp[f(W_1 y + W_2z)|y]$ and $f(x) \approx f_\mathrm{AS}(W_1^\top x)$, where
\begin{align*}
     \Esp[f(W_1 y + W_2z)|y] = \int_\mathcal{Z} f(W_1 y + W_2 z)p_{Z|Y}(z) \dd z,
\end{align*}
and $\mathcal{Z} = \qty{z\in\Real^{D-d} : z = W_2^\top x, x \in X}$,
$p_{Z|Y}$ being the conditional density of $z$ knowing $y$. The joint probability density is given by $p_{YZ}(y,z) = \rho(W_1y + W_2z) $, and the integral is simple to approximate by a MC estimator, namely
\begin{align}
    f_\mathrm{AS}(y) \approx f_\mathrm{MC}(y) =\frac{1}{N} \sum_{i=1}^N f(W_1y+W_2 z_i),
    \label{eq:approxAS}
\end{align}
where $z_i \sim p_{Z|Y}$ is generated numerically using a MC method~\cite{Romor_Tezzele_Rozza_2021}. Often, very few samples are needed because $f$ does not vary much on $\mathcal{Z}$ by construction~\cite{Constantine_Dow_Wang_2014}. An illustration of a $d=1$ active subspace in a $D=3$ full space is shown at \cref{fig:activesubspace}.
\begin{figure}
    \centering
    \includegraphics[width=0.3\textwidth]{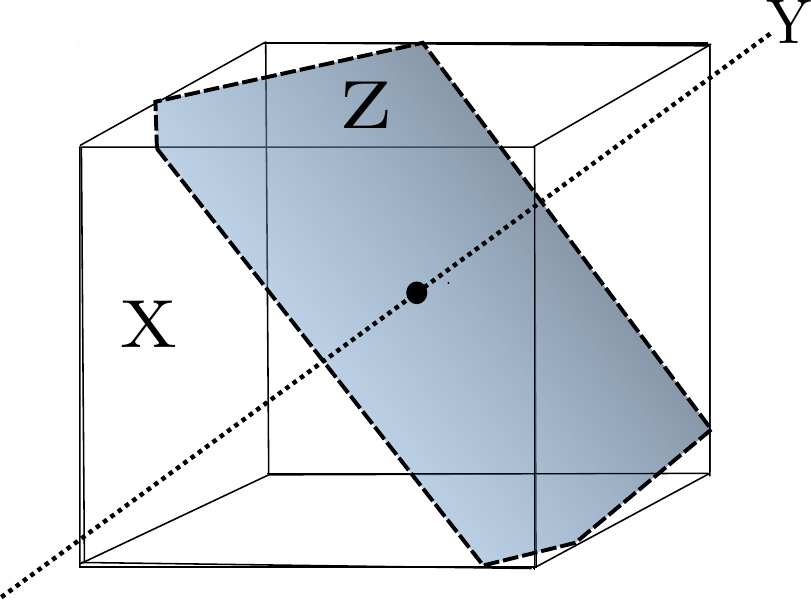}
    \caption{The active subspace $\mathcal{Y}\subset \Real$ in $X=[-1,1]^3$, with the inactive subspace $\mathcal{Z}\subset \Real^2$. The black dot represents the origin.}
    \label{fig:activesubspace}
\end{figure}
Since $f_\mathrm{MC}$ relies on many evaluations of $f$, a response surface or \textit{surrogate} is fitted on the MC estimator, making it possible to approximate it everywhere in $\mathcal{Y}$ with no objective evaluations (after the regression phase). The training data ($S^{(\beta)}$ samples) used to fit the surrogate is noted by 
\begin{align*}
    \mathcal{D}=\qty{y_\mathrm{tr}\in\Real ^{d\times s},f_\mathrm{tr}\in\Real^s : (f_\mathrm{tr})_i = f_\mathrm{MC}((y_\mathrm{tr})_i)}.
\end{align*}
A GPR is used, meaning that at any point in space, the surrogate will follow a space-dependent normal distribution, \textit{i.e.} $f_\mathrm{GPR}(y) | \mathcal{D}\sim \mathcal{N}(\mu(y),\sigma^2(y))$. By using a bayesian approach, one can consider the fitting process as iteratively refining a probabilistic prior on every training data point. A covariance kernel has to be chosen initially, each giving rise to a different model. A classical choice is the squared exponential kernel~\citep{Lukaczyk_Constantine_Palacios_Alonso_2014}, 
\begin{align*}
    k(y_1,y_2) = \exp(-\theta\norm{y_1-y_2}^2),
\end{align*}
where $\theta>0$ is a real parameter to be fitted to the training data. 

A result from~\cite{Constantine_2015} gives an upper bound on the mean squared error (MSE) $\norm{\cdot}_{2,\rho} = \sqrt{\int_X (\cdot)^2 \rho(x) \dd x}$ between $f$ and $f_\mathrm{GPR}$.

\begin{theorem}
Given that:
\begin{itemize}
    \item the 2-norm error on the matrix $W$ computed by SVD is smaller than $\delta_1 >0$;
    \item the MSE between $f_\mathrm{MC}$ and $f_\mathrm{GPR}$ is smaller than $C_2\delta_2$, where $C_2=C_2(X,\rho)$ and $\delta_2=\delta_2(s)$,
\end{itemize}
then $\norm{f - f_\mathrm{GPR}}_{2,\rho}$ is bounded above by
\begin{align*}
    C_1\qty(1+\frac{1}{\sqrt{N}})\qty(\delta_1\sqrt{\sum_{i=1}^d \lambda_i}+\sqrt{\sum_{i=d+1}^D \lambda_i}) + C_2\delta_2,
\end{align*}
where $C_1=C_1(X,\rho)$.
\label{thm:errbound}
\end{theorem}

As mentionned in~\cite{Constantine_Dow_Wang_2014}, this method will interpolate the training data points, which might be somewhate problematic, because the real values $f(x)$ are only approximated by $f_\mathrm{MC}(W_1^\top x)$. Thus, it would be reasonable to have nonzero variance in the GPR, even at the training points. To allow this, a noise term is added to the response surface. The noisy GPR surrogate is 
\begin{align*}
    f_\mathrm{NGPR}(y) = f_\mathrm{GPR}(y) + \nu,
\end{align*}
with an i.i.d. zero-mean Gaussian field
\begin{align*}
    \nu(y) = \nu \sim \mathcal{N}(0,\sigma^2)
\end{align*}
acting as a noise term.
The noisy GPR predictor distribution, knowing $y$, is~\citep{Rasmussen_Williams_2005}
\begin{align*}
    f_\mathrm{NGPR}(y)|\mathcal{D} \sim \mathcal{N}(\mu(y),\Sigma(y)),
\end{align*}
where
\begin{align}
    \mu(y) &= k(y,y_\mathrm{tr})(k(y_\mathrm{tr},y_\mathrm{tr})+\sigma^2I)^{-1}f_\mathrm{tr},\label{eq:mean}\\
    \Sigma(y) &= k(y_\mathrm{tr},y_\mathrm{tr}) - k(y,y_\mathrm{tr})(k(y_\mathrm{tr},y_\mathrm{tr})+\sigma^2I)^{-1}k(y_\mathrm{tr},y).\nonumber
\end{align}

We have to make an assumption that will be useful for the next section.
\begin{assumption}
    The matrix
    $M = (k(y_\mathrm{tr},y_\mathrm{tr})+\sigma^2I)^{-1}$ satisfies
    \begin{align*}
        \sum_{k=1}^s m_{ik} \geq 0 \quad \forall i\in\qty{1,\dots,s},
    \end{align*}
    \textit{i.e.} the sum along every row of $M$ is positive.
    \label{thm:possum}
\end{assumption}

\subsection{Active subspace optimization}
The optimization in the active subspace presented in~\cite{Lukaczyk_Alonso_Constantine_Kochenderfer_Kroo_2015} solves the problem \eqref{eq:nlproblem} by replacing $F$ and $G$ by surrogate models $F_\mathrm{NGPR}$ and $G_\mathrm{NGPR}$ with active subspace mappings $U_1^\top$ and $W_1^\top$ respectively, and solving two problems :
\begin{align}
    \text{Find } (y_F^*,y_G^*)\qq{s.t.} &F_\mathrm{NGPR}(y_F^*) = \min_{y_F\in \mathrm{range}(U_1^\top)} F_\mathrm{NGPR}(y_F),
    \nonumber\\
    &G_\mathrm{NGPR}(y_G^*)\leq0,\label{eq:surroptim}\\
    &c(y_F^*,y_G^*) = 0,\nonumber
\end{align}
and
\begin{align}
    \text{Find } x^*\in X \qq{s.t.} &\norm{x^*}^2 = \min_{x\in X} \norm{x}^2\nonumber\\
    &U_1^\top x = y_F,\label{eq:pullback}\\
    &W_1^\top x = y_G.\nonumber
\end{align}
Problem \eqref{eq:surroptim} aims to find a minimum in the active subspace of $F$ that will also satisfy the constraint in the active subspace of $G$. The constraint $c$ ensures that both points are projections of the same feasible full-space point $x^*$, through the resolution of the quadratic \textit{pullback} problem \eqref{eq:pullback}. This constraint is more precisely defined as the feasibility of \eqref{eq:pullback},
\begin{align*}
    c(y_F,y_G) = \norm{y_F - U_1^\top x^*} + \norm{y_G - W_1^\top x^*} + \mathrm{dist}(x^*,X).
\end{align*}

\section{Methodology}

A few authors have worked before on using GPR with an additional bias. In~\cite{Park_Borth_Wilson_Hunter_Friedersdorf_2022} constant and variable bias models are used to ensure a better fitting of data with many outliers. In~\cite{Zhang_Qiang_Liu_Zhu_Lv_2022} the authors used an approach similar to what we want to achieve (make the approximation conservative to a certain probability), but supposed a very simple Gaussian model for the error, which does not apply in our situation. Finally,~\cite{Zhao_Choi_Lee_Gorsich_2013} uses a non-constant bias based on certain information criteria to ensure conservativeness even in regions with only a few samples. In our case, we will opt for a simple method based on constant bias, but those techniques could be implemented as well to improve our method.

\subsection{Biased ASM}

Adding a positive bias to the estimator $f_\mathrm{MC}$, such that $f^{(\beta)}_\mathrm{MC}(y) = f_\mathrm{MC}(y) + \beta$ with $\beta>0$, will give different training data,
where every function evaluation is shifted upwards, or $f^{(\beta)}_\mathrm{tr} = f_\mathrm{tr} + \beta\vec{\mathbf{1}}$. This will also shift the mean of the GPR surrogate according to \eqref{eq:mean}, similar to~\cite{Park_Borth_Wilson_Hunter_Friedersdorf_2022}:
\begin{align}
    \mu^{(\beta)}(y) = \mu(y) + \beta k(y,y_\mathrm{tr})(k(y_\mathrm{tr},y_\mathrm{tr})+\sigma^2I)^{-1}\vec{\mathbf{1}}.
    \label{eq:predmean}
\end{align}
We will note $f^{(\beta)}_\mathrm{GPR}$ and $f^{(\beta)}_\mathrm{NGPR}$ the noiseless and noisy GPR approximations trained on the biased data. The next theorem concerns how the bias influences the approximation error.
\begin{theorem}\label{thm:errbound2}
Under the assumptions of \cref{thm:errbound}, $\norm{f-f^{(\beta)}_\mathrm{GPR}}_{2,\rho}$ is bounded above by
    \begin{align*}
    C_1\qty(1+\frac{1}{\sqrt{N}})\qty(\delta_1\sqrt{\sum_{i=1}^d \lambda_i}+\sqrt{\sum_{i=d+1}^D \lambda_i}) + C_2\delta_2 + \beta.
\end{align*}
\end{theorem}
\begin{proof}
    \begin{align*}
        \norm{f-f^{(\beta)}_\mathrm{MC}}_{2,\rho} &= \norm{f-f_\mathrm{MC} - \beta}_{2,\rho} \\
        &\leq \norm{f-f_\mathrm{MC}}_{2,\rho} + \norm{\beta}_{2,\rho}. 
    \end{align*}
    According to~\cite[theorem 4.7]{Constantine_2015}, the first term is bounded by 
    \begin{align*}
        \kappa : = C_1\qty(1+\frac{1}{\sqrt{N}})\qty(\delta_1\sqrt{\sum_{i=1}^d \lambda_i}+\sqrt{\sum_{i=d+1}^D \lambda_i}),
    \end{align*}
    and the second term is equal to $\beta$. Under the second hypothesis of \cref{thm:errbound}, 
    we get 
    \begin{align*}
        \norm{f-f^{(\beta)}_\mathrm{GPR}}_{2,\rho} &\leq \norm{f-f^{(\beta)}_\mathrm{MC}}_{2,\rho} + \norm{f^{(\beta)}_\mathrm{MC}-f^{(\beta)}_\mathrm{GPR}}_{2,\rho}
        \leq \kappa + \beta + C_2\delta_2.
    \end{align*}
\end{proof}

\Cref{thm:errbound2} informs us that the approximation error will increase linearly with the bias. This reinforces the idea that the bias must me somewhat minimal even though we want it to be nonzero; there must be a balance between conservativeness and exactness.

We want the surrogate to be greater than the real function for all $x \in X$, meaning the inactive variable must be taken into account as well :
\begin{align*}
    f^{(\beta)}_\mathrm{NGPR}(y) \geq f(W_1 y + W_2 z),
\end{align*}
for any $z\in \mathcal{Z}$ such that $W_1 y + W_2 z \in X$. Fortunately, we wont use latent space sampling from the GPR distribution, and simply use its mean $\mu^{(\beta)}(y)$ as a deterministic surrogate.

Denote the signed distance random variable by $S^{(\beta)} = \mu^{(\beta)}(W_1^\top X) - f(X)$ with $X\sim \rho$, or equivalently,
\begin{align*}
    S^{(\beta)} = \mu^{(\beta)}(Y) - f(W_1Y + W_2Z)
\end{align*}
where $Y\sim p_Y$ and $Z\sim p_{Z|Y}$. For the approximation to be conservative a.s., $S^{(\beta)}$ would have to be a positive random variable. As discussed before, we relax this requirement, and we try to determine a bias value $\beta$ that will make $S^{(\beta)}$ positive with \textit{high-probability}, meaning that by tuning $\beta$ such that $\Prb[S\geq 0] \to 1$ as $\beta\to \infty$, we can achieve a ``conservativeness'' probability as close to 1 as we want. The following results justifies this method.

\begin{theorem}
    Consider $\psi(\beta) = \Prb[S^{(\beta)} > 0]$. Under \cref{thm:possum} and sufficient regularity of $f$, $k$ and $\rho$,
    \begin{itemize}[$-$]
        \item $\psi(\beta)$ is increasing;
        \item for any $\beta$ large enough, $\psi(\beta) = 1$;
        \item $\psi(\beta)$ is continuous.
    \end{itemize}
    \label{thm:increase}
\end{theorem}
\begin{proof}
    Take $0<\beta_1 < \beta_2$ with $\Delta\beta = \beta_2 - \beta_1$. Then, by equation \eqref{eq:predmean},
    \begin{align*}
        \Prb\qty[S^{(\beta_2)}>0] = \Prb\qty[\mu^{(\beta_2)}(W_1^\top X) - f(X) >0]\\
        = \Prb\qty[S^{(\beta_1)} + \underbrace{\Delta\beta k(y,y_\mathrm{tr})(k(y_\mathrm{tr},y_\mathrm{tr})+\sigma^2I)^{-1}\vec{\mathbf{1}}}_{R} >0].
    \end{align*}
    By \cref{thm:possum}, $R$ is positive. Thus, $\Prb[S^{(\beta_1)} + R > 0] \geq \Prb[S^{(\beta_1)} > 0]$ and $\Prb[S^{(\beta_2)} > 0] \geq \Prb[S^{(\beta_1)} > 0]$. Moreover, by taking $\beta > \frac{\norm{\mu}_\infty + \norm{f}_\infty}{k(W_1^\top X,y_\mathrm{tr})(k(y_\mathrm{tr},y_\mathrm{tr})+\sigma^2I)^{-1}\vec{\mathbf{1}}}$, $S^{(\beta)} >0$ everywhere by \eqref{eq:predmean}, making $\psi(\beta) = 1$. Finally, recall that
    \begin{align*}
        \Prb[S^{(\beta)}>0] = \Esp[\mathbbm{1}_\qty{S^{(\beta)}>0}] = \int_{-\infty}^\infty \int_{-\beta v}^{\infty} f_{UV}(u,v)\dd u \dd v, 
    \end{align*}
    where $U = \mu(W_1^\top X) - f(X)$ and $V = k(W_1^\top X,y_\mathrm{tr})(k(y_\mathrm{tr},y_\mathrm{tr})+\sigma^2I)^{-1}\vec{\mathbf{1}}$, so that $S^{(\beta)} = U + \beta V$. By Leibniz integral theorem, $\psi(\beta)$ is continuous as long as $f_{UV}$ is sufficiently smooth.
\end{proof}

\subsection{Probability estimation using Chernoff bounds}
The first method proposed to select the appropriate bias will start by finding $\beta$ such that the expected value of $S^{(\beta)}$ is positive. The following result shows that such a value exists.

\begin{theorem}
    Under \cref{thm:possum}, for every $\varepsilon>0$, there exists a bias $\beta=\Omega(\varepsilon)$ such that $\Esp\qty[S]\geq \varepsilon$.
    \label{thm:bias}
\end{theorem}
\begin{proof}
By linearity,
\begin{align*}
    \Esp[S] = \Esp[\mu^{(\beta)}(Y)] - \Esp[f(W_1Y + W_2Z)].
\end{align*}
By the law of total expectation applied to both terms,
\begin{align*}
    \Esp[S] &= \Esp_Y[\Esp[\mu^{(\beta)}(Y)|Y=y]] - \Esp_Y[\Esp[f(W_1Y + W_2Z)|Y=y]]\\
    &=\Esp_Y[\mu^{(\beta)}(y)] - \Esp_Y[f_\mathrm{AS}(y)].
\end{align*}
The following condition must be satisfied by $\beta$ for the result to be true:
\begin{align*}
    \Esp_Y[\mu(y)] + \beta\Esp_Y[k(y,y_\mathrm{tr})(k(y_\mathrm{tr},y_\mathrm{tr})+\sigma^2I)^{-1}\vec{\mathbf{1}}] - \Esp_Y[f_\mathrm{AS}(y)] \geq \varepsilon,
\end{align*}
which is satisfied for
\begin{align*}
    \beta \geq \frac{\varepsilon - \Esp_Y[\mu(y)-f_\mathrm{AS}(y)]}{\Esp_Y[k(y,y_\mathrm{tr})(k(y_\mathrm{tr},y_\mathrm{tr})+\sigma^2I)^{-1}\vec{\mathbf{1}}]} .
\end{align*}
\Cref{thm:possum} ensures that the denominator is positive.
\end{proof}

This theorem ensures that by choosing a large enough bias, the expected value of $S^{(\beta)}$ will be positive. Moreover, the probability that $S^{(\beta)}$ is at least $\varepsilon$ away from its expected value can be computed using concentration inequalities or tail bounds, notably the Markov, Chebyshev and Chernoff inequalities~\cite{Boucheron_Lugosi_Massart_2013}. We mention the latter as it will be useful later on.
\begin{theorem}[Chernoff Bound]
    Let $X$ be a r.v. with finite moments and $M_X(t) = \Esp\qty[e^{tX}]$ its moment generating function. For all $t>0$,
    \begin{align*}
        \Prb[X>t] \leq \inf_{u>0} M_X(u)e^{-ut}.
    \end{align*}
\end{theorem}

The Chernoff inequality has an advantage over the Chebyshev inequality because it is generally much tighter since the decay can be exponential instead of algebraic. The Chernoff bound will allow us to determine a criterion to increase iteratively $\varepsilon$ until the probability of being conservative is as high as desired. Indeed, since by \cref{thm:bias} we can pick $\varepsilon$ such that $\Esp[S]\geq \varepsilon$, we can try to ensure that $S^{(\beta)}$ is at most $\varepsilon$ away from its expected value with high-probability:
\begin{align}
    \Prb[\abs{S-\Esp[S]}>\varepsilon] &\leq \inf_{u>0}\Esp\qty[e^{u\abs{S-\Esp[S]}}]e^{-u\varepsilon}\nonumber\\
    &\approx \inf_{u>0}\frac{1}{n}\sum_{k=1}^n \exp{u\qty(\abs{s_k-\Esp_\mathrm{boot}[S]}-\varepsilon)}\nonumber\\
    &= \chi(\varepsilon)\label{eq:approxMoment}
\end{align}
for $s_k$ sampled from $S^{(\beta)}$, and $\Esp_\mathrm{boot}[S]$ a bootstrap estimator for the expectation of $S^{(\beta)}$. Then, $\Prb[\abs{S-\Esp[S]}<\varepsilon] \geq 1 - \chi(\varepsilon)$. By setting a tolerance on the probability of success, \textit{e.g.} $\tau = 0.95$, we can iteratively increase $\varepsilon$ until $1-\chi(\varepsilon)\geq \tau$. The probability will eventually converge to 1, as increasing $\varepsilon$ should not increase the variance of $S^{(\beta)}$, thus the coefficients in the exponentials should all eventually become negative and $\chi(\varepsilon) = 0$ as $\varepsilon\to \infty$. Also, note that $1-\chi(\varepsilon)$ is underestimating $\Prb[S^{(\beta)}>0]$, which we will analyze the impact of in \cref{sec:numerical}.

In practice, we will not make use of the parameter $\varepsilon$, but work directly on $\beta$. To obtain a bias value that will satisfy tightly the success probability tolerance, a bisection algorithm~\cite{Quarteroni_Saleri_Gervasio_2014} will be applied on $\beta$ until $1-\chi(\Esp_\mathrm{boot}[S])$ is close enough to $\tau$ (\cref{algo:casm}). The initial bias value has to be taken large enough such that $\Esp_\mathrm{boot}[S]$ is positive, then taking $\varepsilon = \Esp_\mathrm{boot}[S]$ makes sense to ensure conservativeness.

\begin{algorithm}
\textbf{Input: }large $\beta_m$ upper bound for the bias, $f$ the function to be approximated, $\tau$ success probability tolerance, $\delta$ iterative tolerance
\BlankLine
Compute $f_\mathrm{NGPR}$ using standard ASM\;
$n=0$\;
$\beta_l = 0$\;
$\beta_0 = \beta_m / 2$\;
\While{$\abs{1-\chi(\Esp_\mathrm{boot}[S]) -\tau}>\delta$}{
Compute $f^{(\beta)}_\mathrm{NGPR}$ with $\beta_n$\;
Generate a large sample from $S^{(\beta)}$\;
Compute $\Esp_\mathrm{boot}[S]$\;
\If{$\Esp_\mathrm{boot}[S]\leq 0$}{
$\beta_l = \beta_n$\;
$\beta_{n+1} = (\beta_l + \beta_m)/2$\;
}
\Else{
\If{$1-\chi(\Esp_\mathrm{boot}[S]) > \tau$}{$\beta_{m} = \beta_n$\;}
\Else{$\beta_l = \beta_n$\;}
$\beta_{n+1} = (\beta_l + \beta_m)/2$\;
}
$n\leftarrow n+1$\;
}
\Return{$f^{(\beta)}_\mathrm{NGPR}$}\;
\caption{Conservative Active Subspace Method (CASM) - Tail bounds approach}
\label{algo:casm}
\end{algorithm}

\Cref{thm:increase} ensures \cref{algo:casm} converges to the unique bias value yielding the success probability $1-\chi(\Esp_\mathrm{boot}[S])$, under the condition that $\beta_m$ is large enough so $\Prb[S^{(\beta_m)} >0 ] = 1$ and that $\tau$ is not lower than the degree of conservativeness given by $f_\mathrm{NGPR}$.

\subsection{Probability estimation using bootstrap}
\cite{Picheny_Kim_Haftka_Peters_2006} uses bootstrapping to estimate some scalar range for parameters to be higher than a given failure probability, at an arbitrary high user tolerance. Here, we also use bootstrap to estimate the probability of a failure, but to determine the appropriate bias to apply to our surrogate.

An important remark is that the Chernoff inequality might underestimate the conservativeness probability. Thus, we propose a second approach to estimate the desired quantity using Monte Carlo integration. Recall that
\begin{align*}
    \Prb[S>0] = \Esp[\mathbbm{1}_\qty{S>0}],
\end{align*}
where $\mathbbm{1}_U$ is the indicator function of the set $U$. By generating $\qty{s_k : k=1,\dots,K}$ a large sample from $S^{(\beta)}$, we can, by resampling, estimate
\begin{align*}
    \Esp[\mathbbm{1}_\qty{S>0}] \approx  \Esp_\mathrm{boot}[\mathbbm{1}_\qty{S>0}] = \frac{1}{B}\sum_{b=1}^B \frac{\#\qty{s^{(b)}_k > 0 : k\leq K}}{K}.
\end{align*}

\begin{algorithm}
\textbf{Input: }large $\beta_m$ upper bound for the bias, $f$ the function to be approximated, $\tau$ success probability tolerance, $\delta$ iterative tolerance
\BlankLine
Compute $f_\mathrm{NGPR}$ using standard ASM\;
$n=0$\;
$\beta_l = 0$\;
$\beta_0 = \beta_m / 2$\;
\While{$\abs{E_\beta -\tau}>\delta$}{
Compute $f^{(\beta)}_\mathrm{NGPR}$ with $\beta_n$\;
Generate a sample from $S^{(\beta)}$\;
Compute $E_\beta = \Esp_\mathrm{boot}[\mathbbm{1}_\qty{S>0}]$\;
\If{$E_\beta > \tau$}{$\beta_{m} = \beta_n$\;}
\Else{$\beta_l = \beta_n$\;}
$\beta_{n+1} = (\beta_l + \beta_m)/2$\;
}
$n\leftarrow n+1$\;
\Return{$f^{(\beta)}_\mathrm{NGPR}$}\;
\caption{Conservative Active Subspace Method (CASM) - Sampling approach}
\label{algo:casmsamp}
\end{algorithm}
\Cref{algo:casmsamp} converges under the same restrictions as \cref{algo:casm}.

\subsection{Estimation of $S^{(\beta)}$}

Since the evaluation of $f$ might be extremely expensive (\textit{e.g.} high-fidelity aerodynamics simulation), we want to limit the number of times the algorithm has to evaluate $f$. Unfortunately, $S^{(\beta)}$ depends directly on this function, so sampling from $S^{(\beta)}$ ``on the fly'' would be prohibitively expensive. Since $f$ has to be evaluated already many times for the fitting of the surrogate $f_\mathrm{NGPR}$, we wish to reuse those values to estimate the density of $S^{(\beta)}$. To improve the quality of the estimations, we rely on bootstrapping.

We start by sampling $x_k$ uniformly over $X$. The training data for $f_\mathrm{NGPR}$ is obtained by $f_\mathrm{MC}$, as in \eqref{eq:approxAS}, which relies on evaluations of $f(W_1 y + W_2 z)$ at some points $(y_k,z_{ik})$ where $y_k = W_1^\top x_k$, and $z_{ik}$ are sampled from $p_{Z|Y}(z|y_k)$. We tabulate these values as $f_{ik} = f(W_1 y_k+W_2 z_{ik})$. This allows us to generate samples from $S^{(\beta)}$ of the form 
\begin{align*}
    s_{k} &= \mu^{(\beta)}(y_k) - f_{ik},
\end{align*}
where $i$ is picked at random (uniformly) within the index value range. The bootstrap estimators needed are obtained by resampling on these values.

\section{Numerical Experiments}
\label{sec:numerical}

This section will address the  usage of both versions of CASM to estimate optimization constraints. The obtained admissible regions will be analyzed for a synthetic example and the limitations of the methods will be exposed. Then, the method will be tested on a toy thermal design instance.

\subsection{Feasible regions analysis}

The experiments are done with the Python library ATHENA (Advanced Techniques for High dimensional parameter spaces to Enhance Numerical Analysis)~\cite{Romor_Tezzele_Rozza_2021}. This library has all the available tools to do ASM.
As an example, consider the function $G:X \to \Real$ defined as
\begin{align*}
    G(x) = (x_1+2x_2)^2 + x_1 - x_2 - 3,
\end{align*}
with $X=[-1,1]^2$ ($D=2$) and $\rho$ is a uniform distribution over $X$. Using the ASM available within ATHENA, we choose an active subspace of dimension $d=1$, which is spanned by the eigenvector $v_\lambda = \mqty(0.38417958,0.92325839)^\top$, with an eigenvalue of $\lambda = 36.77$ (the eigenvalue associated with the inactive subspace is $3.04$). A visualization of the active and inactive subspaces is available at \cref{fig:GAS}.

\begin{figure}
    \centering
\begin{tikzpicture}[scale=0.9]

\begin{axis}[
colorbar,
colorbar style={title={$G(x)$}},
colormap/viridis,
legend cell align={left},
legend style={fill opacity=0.8, draw opacity=1, text opacity=1, draw=white!80!black},
point meta max=7,
point meta min=-5.24989587671803,
tick align=outside,
tick pos=left,
x grid style={white!69.0196078431373!black},
xlabel={$x_1$},
xmin=-1, xmax=1,
xtick style={color=black},
y grid style={white!69.0196078431373!black},
ylabel={$x_2$},
ymin=-1, ymax=1,
ytick style={color=black}
]
\addplot[forget plot] graphics [includegraphics cmd=\pgfimage,xmin=-1, xmax=1, ymin=-1, ymax=1] {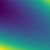};
\addlegendentry{$Z$}
\addplot [thick, black]
table {%
-1 0.4161
1 -0.4161
};
\addlegendentry{$Y$}
\addplot [thick, white]
table {%
-0.4161 -1
0.4161 1
};

\end{axis}

\end{tikzpicture}
    \caption{The active (white) and inactive (black) subspaces of the function $G$.}
    \label{fig:GAS}
\end{figure}

Then, to fit a GP on the data, the library GPy~\cite{gpy2014} is used. On \cref{fig:gprfittingA}, we can see a sampling of values of $G$ projected over the active subspace, and on \cref{fig:gprfittingB}, the resulting GPR over the active subspace with its uncertainty.

\begin{figure}
    \centering
    \subfloat[Projected sample.]{
\begin{tikzpicture}[scale=0.7]

\begin{axis}[
tick align=outside,
tick pos=left,
x grid style={white!69.0196078431373!black},
xlabel={\(\displaystyle y\)},
xmin=-1.31604095320133, xmax=1.33592707765924,
xtick style={color=black},
y grid style={white!69.0196078431373!black},
ylabel={$G(W_1 y + W_2 z)$},
ymin=-4.50014955871495, ymax=5.97907269944079,
ytick style={color=black}
]
\addplot [draw=customOrange, fill=customOrange, mark=*, only marks]
table{%
x  y
-0.896354815018374 2.21936306214879
0.958328873280835 1.27927083671019
1.00475132384632 1.216094299873
-0.00199988531103285 -2.82876301541981
-0.497871695958314 -0.492532067260291
-0.589384125045722 -0.293150432563979
-0.365918488745349 -1.74371168481433
0.21485640998409 -2.93804601346311
0.507949431736204 -3.51654693241953
-0.536497803327693 -0.605793474951522
-0.458754922085547 -0.490693111172097
0.849504392722934 -0.680131468145865
-0.584593646427322 0.116571420682747
0.504290278861399 -2.25047649520336
-0.15774820613213 -2.63573094755911
0.298260386827345 -2.67551491190371
-0.612908178756102 -0.120049066986673
0.977360041603049 0.745700913997944
-0.628403479447617 0.0628590945404068
0.415091113688641 -3.10735377027836
-0.387952110600063 -1.73854948984357
0.113700558900072 -3.39318313421917
-0.895424603889343 2.02416922436358
0.167908514561535 -3.28251959412129
-0.201672337111996 -2.79847499495492
-0.358726995969966 -1.30402738381499
0.521133839540499 -2.90103063155212
-0.125397361979572 -3.35204841560536
0.040465609246799 -3.39485601183398
-0.83483884174653 1.64511920110147
-0.511628463294851 -0.530596707995283
0.105863620792241 -3.40069900725959
-0.469675975441869 -1.25159059130108
-0.577214283651354 -0.301196620774643
1.21538307625648 3.7217311394702
-0.0164131062361433 -3.16916075953824
0.27384913493043 -3.2867646580569
1.05185118007394 1.80174429551333
0.775149098024144 -1.24207025790034
-0.373041408762894 -1.87287490204434
-0.434268926330669 -1.85484684083206
0.36279249979727 -2.94448524884829
0.182945660171615 -3.75325404425542
-0.593691657076913 -0.445339514887755
0.111998665535661 -3.73026556589926
0.209608755366722 -2.98964507088587
0.945884832329903 0.330352351928085
0.202772654216138 -2.4003970374664
-0.3748269979361 -1.67390917830258
0.866737856563326 -0.831598474459232
0.136016478639981 -3.62016542590872
0.164764123525743 -2.87030561424427
0.154424175791631 -3.12039059032064
-0.288111102319706 -1.85291849117085
-0.530610482554637 -0.546100749490299
0.10807864410496 -2.92900471592143
-0.635311999979987 -0.0978722826615251
-0.684104944689398 0.450391752243918
0.783905592785444 -1.13917208645544
-1.1839339280367 5.3576290421284
-0.0160592428507378 -2.76869913185942
-0.465358148546662 -0.712800205542206
0.481675340088552 -2.90149695387648
-0.866015420835607 1.69078971166557
1.09412103543425 1.91851151514207
-0.0860729999116352 -2.64756967133453
0.0198369086059331 -2.67892390549918
-0.676810453636899 0.0485977215669822
-0.561851057917979 -0.172626252097177
-0.00930814400451641 -3.16302881587854
0.622938625531825 -1.13678519197821
0.273821375579376 -3.06052275901985
0.274658026712652 -3.94086214246189
-0.426140222561861 -1.48388787183688
-0.231540315241202 -1.98033861732329
-0.141947883207986 -2.1442067433601
0.0443175812412928 -2.85510378030043
0.686372139119841 -1.27914757490729
0.223011466348999 -3.10634108037721
-0.233388129881782 -2.42181919265053
0.353157931045079 -2.83368888654119
0.903308725234059 0.562528603801301
0.446059520919567 -3.02917032941068
-0.438065357906581 -0.395639830130667
-1.16260752925538 5.04171113377546
0.0503995245586922 -2.59725548723213
-0.150660341975792 -3.01431296197861
0.193561992424352 -4.02382127425332
-0.552537876489086 -0.0348570003288181
0.191235539635831 -3.45537152652943
-0.0373318917818635 -2.95917868759154
-0.629406248753049 0.470074175195
-0.215759320045503 -1.56243637346557
0.0435358869652681 -3.11351418964035
-0.473356674824549 -0.952247879963898
0.178377945333337 -3.20332778579723
0.113971758722043 -3.04745953180706
0.363775627365904 -3.05211024113876
-0.492621006400747 -0.633866342416072
0.809891910460584 -0.772799123558557
-0.799426303130836 1.11899093193062
-0.00588254731324767 -2.85598692016639
0.000674586459264048 -2.83766044437845
-0.46423926913939 -0.874965641365952
0.612037549699219 -2.57485844373961
0.387647665771969 -3.56749318243993
-0.0382170699942074 -3.05319637315975
0.368406597832078 -3.33250193453465
-0.0309756971797778 -3.02112397273751
-0.0153270546223909 -2.85763297240107
-0.156494003603603 -2.38238206784265
-0.147665905441283 -3.10051295421231
-0.227185019062489 -2.57638971719377
0.473450778545699 -3.83044614727785
-1.19549695179858 5.50274441497917
-0.424755836708129 -1.1643574944492
-1.09681583595625 4.2734205706414
0.594376626131925 -1.82263300493664
-0.160275294704161 -2.42457467574241
0.617148784264985 -1.97890476545644
-0.771540393204139 1.10026917136554
1.14466876930252 2.81861619008238
-0.668261899933057 0.31093145941697
0.782097999648416 -1.36952842844853
0.652286773931626 -1.46867986786581
0.841897959878806 -0.552769848914778
0.932543899237894 0.362629906042375
0.055276140201119 -2.71476737631375
0.58406425757283 -2.60446932613087
0.227292010598955 -2.93311789278834
-0.625833836936229 -0.0400366361429132
0.7312937234289 -1.78835076210175
0.528100112284025 -2.56777167624501
0.331866678273653 -3.37797901295465
0.566632729944092 -3.09795935190918
0.344293271309576 -2.35346796385568
0.0547049984338466 -2.90795663148535
-0.0451819466883624 -2.28740914050321
0.217440224093838 -3.63584231011414
0.463587939074872 -2.75524918286356
-0.760845380273871 1.17491873775425
0.445612315304008 -3.00277252148109
0.762804307161574 -1.4527177603764
-0.236473350970921 -2.27503145321372
-0.208665989036195 -2.58411899366968
-0.502226925125006 -0.960198383235663
-0.492486012739501 -1.15231093619524
0.195440326403295 -3.1818759639846
-0.193159252045625 -2.20555874729704
-0.329136554730203 -2.34569034311863
};
\end{axis}

\end{tikzpicture}\label{fig:gprfittingA}}
    \quad
    \subfloat[Noisy GPR fit.]{\input{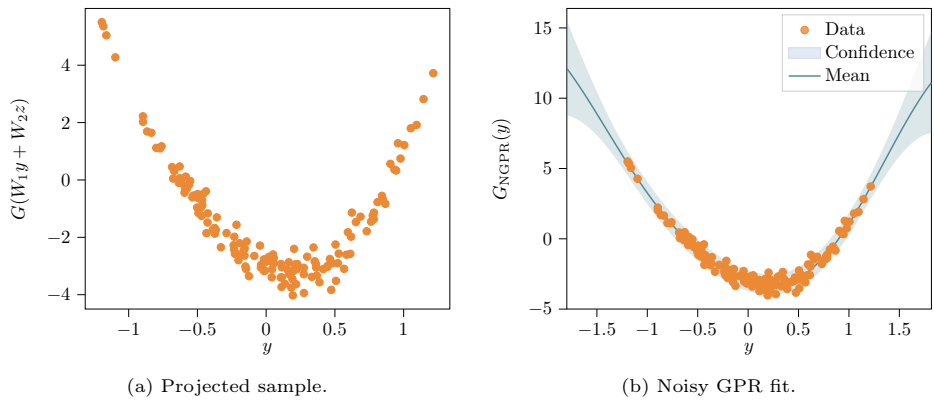}\label{fig:gprfittingB}}
    \caption{(a) Sampling of $G$ projected over the active subspace. (b) The GPR fitted over the active subspace.}
    \label{fig:gprfitting}
\end{figure}

We will restrict ourselves to the case where $G$ would be an inequality constraint in an optimization problem, \textit{i.e.} $G(x)\leq 0$, and analyze the feasible set of the conservative GPR. The resulting feasible sets can be seen on \cref{fig:cons}. The surrogate models for the Chernoff bound and bootstrap approaches were obtained after 7 and 4 bisection iterations respectively, with $\tau=0.95$, $\beta_m = 10$ and $\delta = 0.01$. It can be seen that using the original surrogate $G_\mathrm{NGPR}$ could have led to some optimal but unfeasible points, which is not the case with the biased GPR model, because it is completely enclosed within the original admissible space (for the first approach). 

\begin{figure}
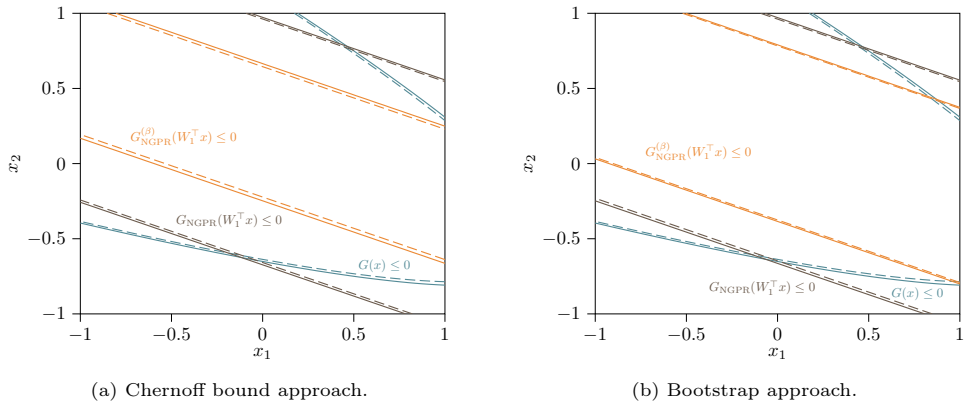

    \centering
    \subfloat[Chernoff bound approach.]{\input{cons}\label{fig:consChernoff}} \qquad 
    \subfloat[Bootstrap approach.]{\input{consboot}\label{fig:consBoot}}
    \caption{Constraints to be enforced when pulling back the solution to the fullspace. The exact constraints are in black, the GPR constraints are in red and the biased constraints (with $\tau = 0.95$) are in blue. The dashed lines indicates the inside of the feasible region.}
    \label{fig:cons}
\end{figure}

\subsection{Limitations}

For functions with too much variation outside of the active subspace, to ensure that $G^{(\beta)}_\mathrm{NGPR}(y) \geq G(W_1y+W_2z)$ with a high probability, the bias would need to be very high, sometimes so much that $G^{(\beta)}_\mathrm{NGPR}(W_1^\top x)>0$ a.e., making the problem unfeasible. 
The intuition that we can get from this is that a function that is able to be represented quite accurately in the active subspace will also admit a CASM surrogate representing adequately the feasible space.

If the number of training samples is relatively small, the number of samples to construct bootstrap distributions will also be quite small, making the estimation of conservativeness probability less precise. Other methods to further reduce the variance of MC estimators could be used.

We want to study the degree to which the Chernoff bound underestimates the decay of the probability, making constraints too tight and solutions suboptimal. To test this, we sampled 10000 points at random and computed the proportion of points where the approximated constraint overestimates the exact constraint (the conservativeness probability). What we see (\cref{tab:conserv}) is that the Chernoff bound is not tight enough and even for relatively low $\tau$ values, the conservativeness probability of the result is very high. The estimation of the success probability by bootstrap is clearly better. $\tau = 0.25$ does not work for the bootstrap method because the base conservativeness probability ($\beta = 0$) of $G$ is higher than $40\%$. This can be detected automatically if the bias value is very small in proportion to the average $S^{(\beta)}$ value. We also compute the unfeasibility ratio ($UR$), which is the proportion of points violating the exact constraint within the feasible points according to the approximate constraint,
\begin{align*}
UR = \Prb\qty[G(x) \geq 0 \left| G^{(\beta)}_{\mathrm{NGPR}}(x) \leq 0\right.] \approx \frac{\#\qty{G^{(\beta)}_{\mathrm{NGPR}}(x) \leq 0\qq{AND} G(x) \geq 0 } }{ \#\qty{G^{(\beta)}_{\mathrm{NGPR}}(x) \leq 0}}.
\end{align*}
As seen on \cref{fig:consChernoff}, with $\tau=0.95$, the blue region is completely enclosed within the black region, and no such point exists. For other cases (\cref{fig:consBoot}, \cref{fig:conslowA}), the regions of unfeasibility are very small, with a probability of less than $1\%$ of being infeasible. The $UR$ for the bootstrap approach with $\tau = 0.5$ (\cref{fig:conslowB}) is a bit higher ($5\%$), but that is because the bootstrap estimator of the probability is more precise, thus a conservativeness probability of $50\%$ is quite low and more points violate the constraint. Nevertheless, the method has shown it can ensure that the optimal solution in the active subspace will also satisfy the original constraint when pulled back to the full space in some scenarios.
\begin{table}[]
    \centering
    \begin{tabular}{ccccc}
         & \multicolumn{2}{c}{Observed Cons.} & \multicolumn{2}{c}{$UR$} \\
         \hline
         Target Cons. ($\tau$) & Chernoff & Bootstrap & Chernoff & Bootstrap  \\
         \hline
         \hline
         0.95 & 0.99 & 0.95 & 0.000 & 0.002\\
         \hline
         0.50 & 0.92 & 0.57 & 0.002 & 0.049\\
         \hline
         0.25 & 0.87 & $-$ & 0.009 & $-$\\
         \hline
    \end{tabular}
    \caption{Observed conservativeness and unfeasibility for many $\tau$ values, for both CASM approaches (Chernoff bound and bootstrap).}
    \label{tab:conserv}
\end{table}
\begin{figure}
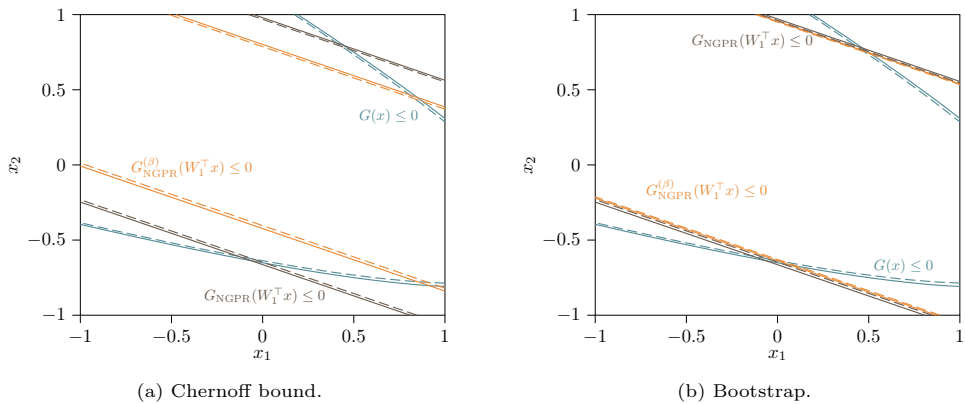

    \centering
    \subfloat[Chernoff bound.]{\input{cons05}\label{fig:conslowA}} \qquad 
    \subfloat[Bootstrap.]{\input{cons025}\label{fig:conslowB}}
    \caption{Feasible sets obtained for $\tau = 0.5$ for the different approaches.}
    \label{fig:conslow}
\end{figure}

\subsection{Two-Material Optimal Distribution Problem}
This section is completed by applying the method to a classical toy problem in computational mechanics, originally credited to Céa and Malanowski \cite{cea_1970}. The \textit{transmission} problem  consists of computing the diffusion of a scalar field $u$ (\textit{e.g.} heat) over a domain $\Omega$, which is partitioned into two subdomains $\Omega_1$ and $\Omega_2$ with different diffusion coefficients $k_1$ and $k_2$, subject to a flux continuity condition at their interface \cite{delfour_2011}:
\begin{align*}
    \begin{cases}
        -k_1 \Delta u = f &\qq{in} \Omega_1\\
        -k_2 \Delta u = f &\qq{in} \Omega_2\\
        \qquad u = 0 &\qq{on} \partial \Omega\\
        k_1\pdv{u}{n_1} + k_2\pdv{u}{n_2}=0 &\qq{on} \partial \Omega_1 \cap \partial \Omega_2,
    \end{cases}
\end{align*}
where $f$ is a heat source distribution and $n_1$ and $n_2$ denote the outward boundary normals for each subdomains. By setting $\theta$ as the binary indicator function of $\Omega_1$, this problem can be reformulated in variational form as
\begin{align}
    \int_\Omega (k_1\theta + k_2(1-\theta))\nabla u \cdot \nabla v \dd x= \int_\Omega fv \dd x
    \label{eq:weakdiffusion}
\end{align}
for every test function $v\in H_0^1(\Omega)$. The design concern arising from this problem is to determine the distribution of the two materials in the domain, \textit{i.e.} determine the indicator function $\theta$, according to some objective and constraints \cite{Thurier_Lesieutre_Frecker_Adair_2019,park_conceptual_2019}. For instance, one could want to minimize the volume of the first material (\textit{e.g.} for economical purposes) while ensuring that the total thermal energy over the domain stays below a given threshold $E_\mathrm{max}$. The thermal energy (or heat compliance) is
\begin{align*}
    E(\theta) = \frac{1}{2}\int_\Omega  (k_1\theta + k_2(1-\theta)) \abs{\nabla u}^2 \dd x,
\end{align*}
and the volume of the first material is 
\begin{align*}
    V(\theta) = \int_\Omega \theta \dd x.
\end{align*}
To solve numerically \eqref{eq:weakdiffusion}, the finite element (FE) method is used. Then, a piecewise constant approximation of $\theta$ is made to solve the optimization problem. The design variable $\theta_i \in [0,1]$ is used to represent the value of $\theta$ over the $i$-th element. Note that the design variables are allowed to take non-integer values in opposition to $\theta$, to simulate for mixture of materials in certain regions too refined to be captured by the resolution of the FE mesh.

We want to solve the following design optimization problem (of the same form as \eqref{eq:nlproblem}) using the CASM approach:
\begin{align}
    \min_{\theta \in [0,1]^D} V(\theta) \qq{s.t.} E(\theta) \leq E_\mathrm{max}\qq{\textit{\&}}\text{eq. \eqref{eq:weakdiffusion}},
    \label{eq:minvol}
\end{align}
where $D$ is the number of elements in the FE mesh. As in \cite{delfour_2011}, $\Omega = [-1,1]^2$, $k_1=2$, $k_2=1$ and $f(x,y) = 56(1-\abs{x}-\abs{y})^6$. Here, the problem is discretized with a structured mesh of $D=512$ triangular second-order Lagrange elements. According to prior numerical experiments, an energy threshold of $E_\mathrm{max}=1.35$ is chosen.

Since the objective is linear, a one-dimensional AS would be sufficient to capture all of the function's variation. To know how many dimensions are necessary to approximate accurately the energy constraint, an analysis of the eigenvalues of the covariance matrix is necessary. First, a rescaling is applied to map $[-1,1]^D$ to $[0,1]^D$, because the AS theory is applicable only to the former. On \cref{fig:eigenEnergy}, we see that after the first eigenvalue, there is a significant drop in the magnitude of the spectrum. We thus choose $d=1$ for the dimension of the AS. The first eigenvector is displayed on \cref{fig:eigenvector}. The annular region around the center seems to influence more the heat compliance of the design than the rest of the domain.

\begin{figure}
    \centering
    \subfloat[10 largest eigenvalues.]{\begin{tikzpicture}[scale=0.7]
  \begin{semilogyaxis}[
    xlabel={$i$},
    ylabel={$\lambda_i$}
    ]
    \addplot[mark=*, color=customTeal] coordinates {
      (0,0.000882448329127676)
(1,1.5953761821851023e-06)
(2,1.489499783719553e-06)
(3,1.383757980590466e-06)
(4,1.2411339490419155e-06)
(5,1.1333426582558588e-06)
(6,1.0449642267748638e-06)
(7,1.018367092349931e-06)
(8,8.842110106606098e-07)
(9,8.625994793559538e-07)
    };
  \end{semilogyaxis}
\end{tikzpicture}}
    \subfloat[Complete ordered spectrum.]{\input{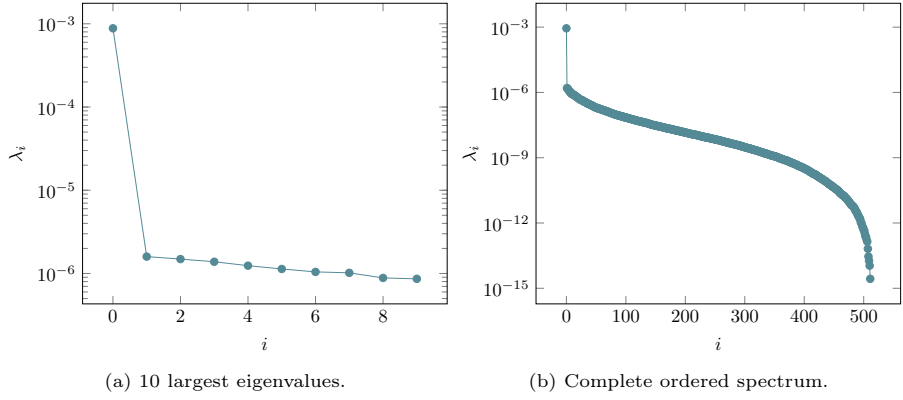}}
    \caption{Eigenvalues of the gradient's covariance matrix for the thermal energy function.}
    \label{fig:eigenEnergy}
\end{figure}

\begin{figure}
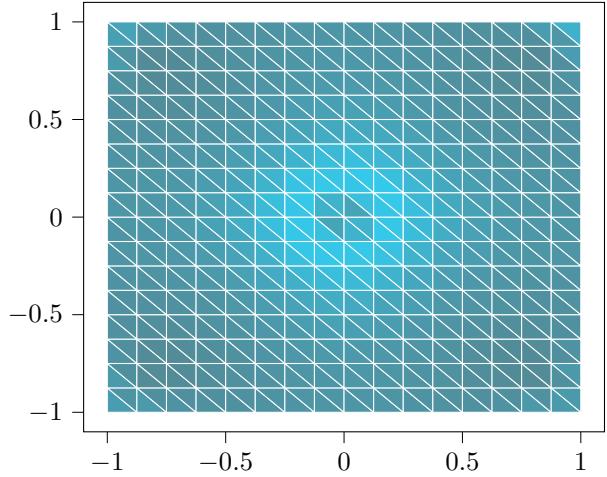

    \centering
    \include{eigenvector}
    \caption{First eigenvector of the covariance matrix. The $\theta_i$ values are ranging from 0.5 (darker) to 0.56 (lighter).}
    \label{fig:eigenvector}
\end{figure}

Then, the profile of the constraint along the chosen active subspace is determined, as shown on \cref{fig:energyprofile}, using 50 samples. The data seems to be well approximated by a linear surrogate model, so this will be used for this case instead of GPR. After applying \cref{algo:casmsamp}, we notice that although it generally does not underestimate significantly the target probability, it is much harder to obtain precise probability estimations for $\tau$ close to 1 for the problem at hand. Further experiments are conducted using \cref{algo:casm}.

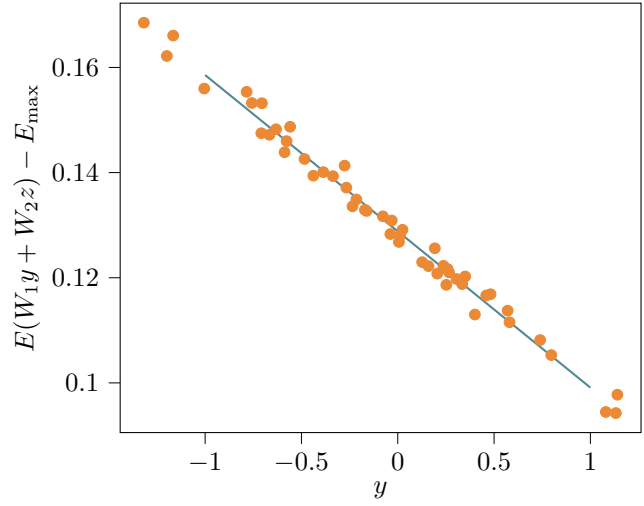
\begin{figure}
    \centering
\begin{tikzpicture}

\definecolor{color0}{rgb}{0.12156862745098,0.466666666666667,0.705882352941177}

\begin{axis}[
tick align=outside,
tick pos=left,
x grid style={white!69.0196078431373!black},
xlabel={\(\displaystyle y\)},
xmin=-1.4410749426897, xmax=1.26297073399089,
xtick style={color=black},
y grid style={white!69.0196078431373!black},
ylabel={\(\displaystyle E(W_1 y + W_2 z)-E_\mathrm{max}\)},
ymin=0.0905476073543092, ymax=0.172226756738602,
ytick style={color=black}
]
\addplot [draw=customOrange, fill=customOrange, mark=*, only marks]
table{%
x  y
-0.17038549517073 0.132906651872292
0.258848676831106 0.121652063151729
-0.438109604149488 0.139409353939485
0.252072345626697 0.118649419522577
0.304791830264462 0.119764791254712
-0.214917902017474 0.134898379140249
-1.16585504977959 0.166063446288133
-0.266730282452341 0.137135980059573
0.796831751123698 0.105288953163834
-0.705396571497278 0.153200896243329
-1.31816377556785 0.168514068130225
1.14005956686904 0.0977518372025712
-0.784558067644902 0.155363830232904
0.159462648262457 0.122207168336316
0.205054879615362 0.120776808823785
-0.587594800207045 0.14387042885749
0.26590152966497 0.121026517822624
0.00606297334119027 0.126805053253108
-0.483962088256118 0.142586023556913
1.1321676534527 0.0942602959626861
-0.66643627377519 0.147192151363414
0.236983511746583 0.122272468938389
-0.757363870404909 0.153242678764843
0.580228916629032 0.111531806286131
0.0246357396185034 0.129128533721324
0.192385416173853 0.125609327703433
0.350256485565497 0.120248195692379
-0.576820860981051 0.14598472191154
-0.0385205225800752 0.128311575163893
-1.19876021757667 0.162179585448306
0.481776496447627 0.11686870869278
-0.161748639309009 0.132711789048346
-0.0771812195121219 0.131690495220655
1.08037886610362 0.0944652584340389
0.012223645243117 0.128012256423258
0.570488663548654 0.113745561079605
-0.558721422955762 0.148711020989989
-0.70842151318309 0.147493702756629
-0.386353914379966 0.140091250225116
0.126772136249031 0.122971285843776
-0.336383774070211 0.13933007201691
-0.632096928399709 0.148246067778219
0.334185125524124 0.118787835080006
-1.00450736905449 0.155987734587488
-0.275915536650108 0.141324098483092
0.459220627207626 0.11663455201053
0.739205183517504 0.108150709165265
-0.0313085263519704 0.130892126569384
0.400349550963802 0.113000186597256
-0.234925881633035 0.13358078763141
};
\addplot [thick, customTeal]
table {%
-1 0.158554416909199
-0.959183673469388 0.157341101612316
-0.918367346938776 0.156127786315433
-0.877551020408163 0.154914471018549
-0.836734693877551 0.153701155721666
-0.795918367346939 0.152487840424783
-0.755102040816326 0.1512745251279
-0.714285714285714 0.150061209831016
-0.673469387755102 0.148847894534133
-0.63265306122449 0.14763457923725
-0.591836734693878 0.146421263940367
-0.551020408163265 0.145207948643483
-0.510204081632653 0.1439946333466
-0.469387755102041 0.142781318049717
-0.428571428571429 0.141568002752834
-0.387755102040816 0.140354687455951
-0.346938775510204 0.139141372159067
-0.306122448979592 0.137928056862184
-0.26530612244898 0.136714741565301
-0.224489795918367 0.135501426268418
-0.183673469387755 0.134288110971534
-0.142857142857143 0.133074795674651
-0.102040816326531 0.131861480377768
-0.0612244897959184 0.130648165080885
-0.0204081632653061 0.129434849784001
0.0204081632653061 0.128221534487118
0.0612244897959182 0.127008219190235
0.102040816326531 0.125794903893352
0.142857142857143 0.124581588596469
0.183673469387755 0.123368273299585
0.224489795918367 0.122154958002702
0.265306122448979 0.120941642705819
0.306122448979592 0.119728327408936
0.346938775510204 0.118515012112052
0.387755102040816 0.117301696815169
0.428571428571428 0.116088381518286
0.469387755102041 0.114875066221403
0.510204081632653 0.113661750924519
0.551020408163265 0.112448435627636
0.591836734693877 0.111235120330753
0.63265306122449 0.11002180503387
0.673469387755102 0.108808489736987
0.714285714285714 0.107595174440103
0.755102040816326 0.10638185914322
0.795918367346939 0.105168543846337
0.836734693877551 0.103955228549454
0.877551020408163 0.10274191325257
0.918367346938775 0.101528597955687
0.959183673469388 0.100315282658804
1 0.0991019673619207
};
\end{axis}

\end{tikzpicture}
    \caption{Constraint function profile when restricted to a one-dimensional AS, with a linear fit.}
    \label{fig:energyprofile}
\end{figure}

\begin{figure}
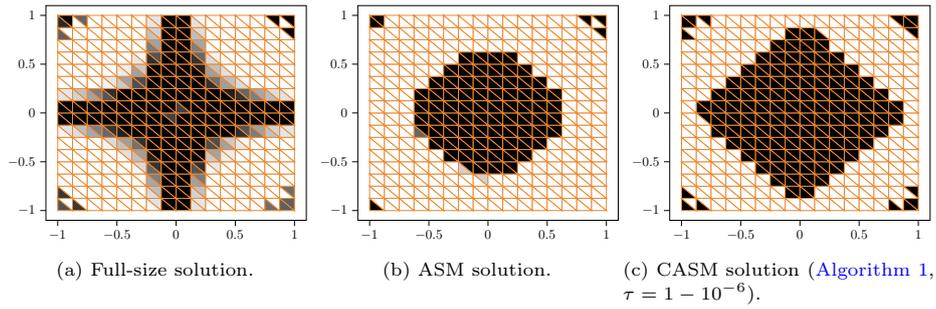

    \centering
    \subfloat[Full-size solution.]{\input{optimaltopEXACT}}
    \subfloat[ASM solution.]{\input{solAS}}
    \subfloat[CASM solution (\cref{algo:casm}, $\tau=1-10^{-6}$).]{\input{solCAS}}
    \caption{Comparison between the material distributions obtained for the full-size problem \eqref{eq:minvol}, the AS reduced problem and the CASM reduced problem. Black ($\theta=1$) and white ($\theta=0$) regions have respectively diffusion coefficients $k_1$ and $k_2$.}
    \label{fig:optimaltopology}
\end{figure}

\begin{table}
    \centering
    \begin{tabular}{cccc}
        Optimization Problem & Bias ($\beta$) & Objective Minimum & Exact Constraint Violation (\%) \\
        \hline\hline Full-size & $-$ & 1.5979978209236378 & $6.53\times10^{-6}$\\
        \hline ASM & 0 & 1.2262997535329518 & 10.25 \\
        \hline CASM \ref{algo:casm} ($i=3$) & 0.07080078125 & 1.5280582174298663 & 4.12\\
        \hline CASM \ref{algo:casm} ($i=4$) & 0.0927734375 & 1.6306459831326832 & 2.62\\
        \hline CASM \ref{algo:casm} ($i=5$) & 0.1171875 & 1.7524308196234444 & 0.19\\
        \hline CASM \ref{algo:casm} ($i=6$) & 0.13916015625 & 1.8673515105314 & 0\\
        \hline
    \end{tabular}
    \caption{Optimal solutions obtained and their feasibility with respect to the exact constraint for the different methods used. CASM \ref{algo:casm} refers to \cref{algo:casm}.}
    \label{tab:optimresults}
\end{table}

We see on \cref{fig:optimaltopology} the comparison between the different optima found. The MMA optimization solver (Method of Moving Asymptotes, originally developed by Krister Svanberg \cite{Svanberg_1987}) from the library NLopt \cite{Jonhson_2022} was used for all cases. The optimal values found and the constraint violations are shown in \cref{tab:optimresults}. Experiments are done for $\tau = 1-10^{-i}$ with tolerances $10^{-i-1}$ for increasing $i$ values. What can be observed is that using the standard AS method for the surrogate, the optimum found is better, but it violates the exact thermal energy constraint by more than 10\%. By using conservative surrogates, there is still a violation, but it is significantly smaller (less than 5\%), meaning that the optima is located on a section of the boundary that is not included in the exact feasible region. This improvement is attained with virtually no additional computational cost when compared to the regular AS method. Indeed, the computation of the bias is extremely quick since it does not require any new exact constraint evaluation. Because of this, and the fact that the optimization problem has been reduced to a one-dimensional problem, it is not expensive to iterate over multiple increasing values of $\tau$ until the constraint is completely satisfied.

As seen in \cref{tab:optimresults}, for $\tau=1-10^{-6}$, the optimal solution found does not violate the exact constraint. This solution is displayed in \cref{fig:optimaltopology}. Although it is suboptimal with respect to the full-size problem, the conservativeness level of the approximation is high enough to ensure that the solution is feasible with respect to the exact constraint. It is a possibility that increasing the number of dimensions in the active subspace would reduce the suboptimality of the solution for the biased problem by allowing for more degrees of freedom, but because of the linear nature of the active subspace method, it would likely fail to approximate correctly an admissible set bounded by a highly nonconvex boundary.

To confirm that the first eigenvector captures correctly the essence of the full-size optimization problem, by imposing a conservativeness probability of $\tau=1-10^{-10}$, we recover the same cross-shaped pattern that was observed in the optimal design for the original problem (\cref{fig:optresulttight}). This means that the arms of the cross have a less significant impact on the heat compliance of the design, but for a much more strict set of feasible designs, their contribution is non-negligible.

\begin{figure}
    \centering
    \input{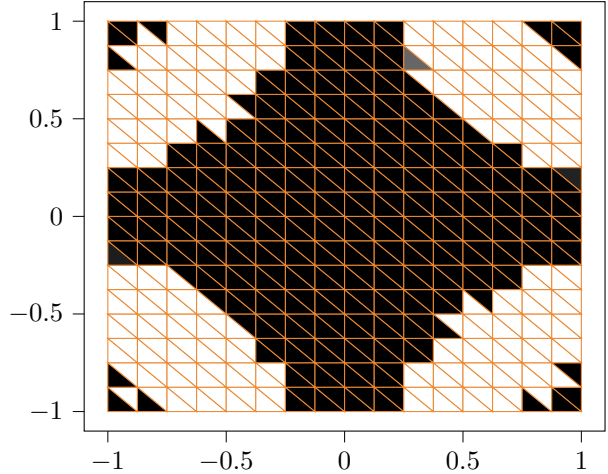}
    \caption{Optimal result for the biased problem with $\tau=1-10^{-10}$.}
    \label{fig:optresulttight}
\end{figure}

\section{Conclusion}
A method to ensure conservativeness of optimization constraints in the ASM framework has been presented. After ensuring that the expected difference between the surrogate and the real constraint is positive, the bias is iteratively refined to make it optimal according to the user desired success probability level by using a bisection algorithm and concentration inequalities. This is in accordance with the theoretical results proven, which indicates that the bias should be high enough to ensure conservativeness, but also not too large to minimize the MSE between the real constraint and the surrogate. 

To estimate statistical quantities of the distribution of $S^{(\beta)}$, bootstrapping is used. The training data for the surrogate is reused in a way to minimize the number of evaluations of expensive simulation models.

After showing that CASM works on a toy problem, we explored its limitation. First, there is a potential of making the problem infeasible if the constraint function has still a lot of variation outside of the active subspace. Then, the concentration inequality used underestimates the real decay of the conservativeness probability, thus making the bias potentially larger than it needs to be to attain the user-defined threshold. As mentionned before, this decreases the quality of the surrogate as an approximation. Future work could try to study better how to accurately bound the tails of $S^{(\beta)}$. Fortunately, the second approach using only bootstrap to estimate the probability of success is more precise and even less costly. 

The method also has shown some usefulness in a concrete design optimization problem, by attenuating exact constraint violations in a significant way. The results could potentially be further improved by using a larger initial training sample to get a more precise estimation for the conservativeness probability, and by allowing for a larger number of dimensions in the active subspace.

All the Python code used to generate the results is available on Github (\texttt{paluneau/CASM}) and archived through Zenodo \cite{Luneau_2024}.

\section*{Acknowledgements}

The author would like to thank Pr. Jean Deteix and Pr. Etienne Marceau for their advice while writing this paper. This work was realized while the author was funded by a graduate scholarship of the \textit{Institut des sciences mathématiques} (ISM) in Montreal and a Natural Sciences and Engineering Research Council of Canada (NSERC) grant under Pr. Jean Deteix [grant number CRDPJ-501716-2016].

\section*{Declaration of Competing Interest}
The authors declare that they have no competing interests (financial or personal) that could have appeared to influence the research presented in this paper.

\bibliography{ref}

\end{document}